\documentclass[12pt]{article}
\usepackage{latexsym}
\usepackage{amssymb}
\usepackage{amsmath}
\usepackage{amsthm}
\usepackage{mathabx}
\usepackage{rotating}
\usepackage{multirow}
\usepackage{mathrsfs}
\usepackage{wasysym}
\usepackage{enumerate}
\usepackage{nicefrac}
\usepackage{rawfonts}
\input{prepictex}
\input{pictex}
\input{postpictex}
\DeclareMathAlphabet{\zap}{OT1}{pzc}{m}{it}
\def\CC{\mathbb C}

\newtheorem{main}{Theorem}

\newtheorem{thm}{Theorem}[section]
\newtheorem*{them}{Theorem}
\newtheorem{lem}[thm]{Lemma}
\newtheorem{prop}[thm]{Proposition}
\newtheorem{cor}[thm]{Corollary}
\def\ro{\mathring{r}}
\newenvironment{xpl}{\mbox{ }\\ {\bf  Example.}\mbox{ }}{
\hfill $\diamondsuit$\mbox{}\bigskip}

\def\RR{{\mathbb R}}
\def\CP{{\mathbb C \mathbb P}}

\DeclareMathOperator{\vol}{Vol}

\begin{document}

\title{Curvature in the Balance:\\
The  Weyl Functional  and Scalar\\ Curvature
of $4$-Manifolds
}

\author{Claude LeBrun\thanks{Supported 
in part by  NSF grant DMS-1906267}\\ 
Stony
 Brook
 University} 
  
\date{}

\maketitle

\hspace{1.8in}
\begin{minipage}{3.5in}
\begin{quote} {\em Dedicated to  my friend, mentor, and esteemed colleague
 Blaine Lawson,  on the occasion  of
his eightieth birthday. }
\end{quote}
\end{minipage}

\bigskip

\begin{abstract}The infimum of the Weyl functional  is shown to  be surprisingly small  on  many compact 
 $4$-manifolds that admit
positive-scalar-curvature metrics. Results are also proved that systematically compare  the
 scalar and self-dual Weyl curvatures of certain  
almost-K\"ahler $4$-manifolds.  \end{abstract}

 The curvature tensor of an  oriented Riemannian $4$-manifold $(M^4,g)$ may be 
 invariantly decomposed into exactly four independent pieces
 $$\mathcal{R} = s\oplus \mathring{r} \oplus W_+\oplus W_-,$$
 where $s$ is the scalar curvature, where $\mathring{r}$ is the trace-free Ricci curvature, 
 and where $W_\pm$ are respectively the self-dual and anti-self-dual Weyl curvatures. 
This happens because   the curvature tensor $\mathcal{R}$ at a point may naturally be thought of as  an element 
 of $\odot^2\Lambda^2 \cap (\Lambda^4)^\perp$, and decomposing 
 this vector space into irreducible   $\mathbf{SO}(n)$-modules splits it into exactly four factors when $n=4$. 
 Four dimensions is completely anomalous in this respect; by contrast, the curvature
consists of  just three invariant pieces when $n> 4$, of  just  two pieces when $n=3$, 
and of only  a single piece when $n=2$. 

Assuming henceforth that $M^4$ is compact  and without boundary, we now obtain four
basic quadratic curvature functionals on the space of Riemannian metrics  on $M$
by taking the $L^2$-norm-squared of each of our  curvature pieces 
\begin{eqnarray*}
g&\longmapsto&\int_M s^2d\mu_g\\
g&\longmapsto&\int_M |\mathring{r}|^2\, d\mu_g\\
g&\longmapsto&\int_M |W_+|^2d\mu_g\\
g&\longmapsto&\int_M |W_-|^2d\mu_g\\
\end{eqnarray*}
and any other quadratic curvature functional is
then  a linear combination of these four. 
Each of these functionals is invariant under {constant} rescalings $g\rightsquigarrow c g$, for $c\in \RR^+$,  
and the last  two  functionals are actually  {\em conformally invariant}, in the sense that they 
are unaltered by arbitrary conformal rescalings $g\rightsquigarrow u g$, where $u: M \to \RR^+$
is a smooth positive function. 

On the other hand, these four functionals are not genuinely independent, because the
$4$-dimensional 
Gauss-Bonnet formula
\begin{equation}
\label{gb}
\chi (M) = \frac{1}{8\pi^2} \int_M \left( \frac{s^2}{24} + |W_+|^2 + |W_-|^2 - \frac{|\mathring{r}|^2}{2} \right) d\mu_g
\end{equation}
and Thom-Hirzebruch signature formula 
\begin{equation}
\label{th}
\tau (M) = \frac{1}{12\pi^2} \int_M \left(  |W_+|^2  -|W_-|^2  \right) d\mu_g
\end{equation}
express two important homotopy invariants of the compact oriented $4$-manifold $M$ as linear combinations
of these  basic  curvature functionals. For metrics on a  fixed oriented $4$-manifold $M$, the two functionals
\begin{equation}
\label{basis}
g\longmapsto\int_M \frac{s^2}{24} d\mu_g , \qquad
g\longmapsto\int_M |W_+|^2d\mu_g , 
\end{equation}
therefore completely determine every  other quadratic curvature functional. 

The main theme of this article concerns a {\sf question of balance}: how do the two functionals \eqref{basis} 
compare in size, for  specific types  of metrics on interesting  classes $4$-manifolds? 

One  source of motivation for this question stems from the K\"ahler case.
Suppose that $g$ is a K\"ahler metric on $(M,J)$, and that $M$ is given the orientation 
 determined by the complex-structure
tensor $J$. We then have the point-wise identity
$$|W_+|^2 = \frac{s^2}{24}, $$
and  our two basic functionals  \eqref{basis} therefore  coincide on K\"ahler metrics. 

Einstein metrics provide a particularly compelling context for this issue. Recall \cite{bes} that
a Riemannian metric $g$ is said to be {\sf Einstein} if its Ricci tensor satisfies $r = \lambda g$ for 
some constant $\lambda$. In any dimension $n > 2$, this is equivalent to requiring  the
trace-free Ricci tensor $\mathring{r} = r -\frac{s}{n}g$ to vanish, and the scalar curvature $s$ of such a metric then coincides
with $n$ times the Einstein constant $\lambda$. When $n=4$, our balance question turns out to be 
highly relevant to the study of Einstein metrics, but the direction  in which the balance tips  critically  depends
on the sign of the Einstein constant.
 For example,  when the scalar curvature is positive, 
the self-dual Weyl curvature almost always  outweighs the scalar curvature \cite{G1,gl2}: 

\begin{them}[Gursky] Let $(M^4,g)$ be a compact oriented Einstein $4$-manifold 
with $s> 0$ that is not an irreducible symmetric space. Then 
$$\int_M |W_+|^2 d\mu_g  \geq \int_M  \frac{s^2}{24}d\mu_g,$$
with equality iff $g$ is a locally K\"ahler-Einstein metric. 
\end{them}

By contrast, in the negative-scalar-curvature setting, there are large classes of $4$-manifolds where the balance tips in the opposite direction
\cite{lmo,lebcam}:

\begin{them}[L]  Let $M$ be a smooth compact $4$-manifold that admits a symplectic form,
but does not admit an Einstein metric with $s> 0$. Then, with respect to the symplectic orientation, 
any Einstein metric $g$
on $M$ satisfies
$$ \int_M  \frac{s^2}{24}d\mu_g \geq \int_M |W_+|^2 d\mu_g ,$$
with equality iff $g$ is a K\"ahler-Einstein metric. 
\end{them}

(Here the  assumption that  the symplectic manifold  $M$ admits an Einstein metric, but does admit
an Einstein metric of positive scalar curvature, guarantees that the symplectic form $\omega$ satisfies
$c_1 \cdot [\omega ] \leq 0$  and  $c_1^2 \geq 0$. A result of Taubes \cite{taubes} then implies that, for the spin$^c$ structure
determined by the symplectic form, 
the unperturbed  Seiberg-Witten equations admit  a solution for every metric, and the desired inequality
then follows from the Weitzenb\"ock formula for the Seiberg-Witten equations. By contrast, the assumption 
in  Gursky's theorem
that  $(M,g)$ is not an irreducible symmetric space    implies, 
by  a result  of Hitchin \cite[Theorem 13.30]{bes}, that $W_+\nequiv 0$; the theorem  is then deduced,  using 
 a clever conformal-rescaling argument, 
from a Weitzenb\"ock formula for $W_+$ that holds for any Einstein $4$-manifold, since the second Bianchi 
identity on such a space implies that    $\delta W_+=0$.) 

\bigskip

Given these  results about the Einstein case, it might therefore seem tempting  to ask  
about the balance between our two basic functionals \eqref{basis} for arbitrary Riemannian metrics on a smooth compact
oriented $4$-manifold. However, this na\"{\i}ve form of the question is just  silly, 
because because $\int |W_+|^2 d\mu$ is conformally invariant, while $\int s^2 d\mu$ varies wildly in any conformal class!

\begin{xpl} Let $(M,g)$ be a compact oriented Riemannian $4$-manifold, and consider arbitrary conformal rescalings
$\hat{g} = u^2 g$, where $u: M \to \RR^+$ is a smooth positive function. The {\sf Yamabe functional} of such a  conformally rescaled metric
is then given by
$$
\mathscr{E}( \hat{g}) := \frac{\int_M s_{\hat{g}}d\mu_{\hat{g}}}{\sqrt{\int_M d\mu_{\hat{g}}}}= 
\frac{\int_M \left[ 6|\nabla u|^2 + su^2 \right] d\mu_g}{\sqrt{\int_M u^4 d\mu_g}}
$$
and, since there are functions $u$ that are $C^0$ close to $1$, but which are wildly oscillatory on a microscopic scale,  one immediately sees
that there are sequences of metrics $\hat{g}_j$ in the given conformal class $[g]: = \{ u^2 g\}$ that remain $C^0$ close to $g$, but have 
$\mathscr{E}(\hat{g}_j) \to +\infty$. But since $\int s^2 d\mu_{\hat{g}} \geq [\mathscr{E}(\hat{g})]^2$ by the Cauchy-Schwarz inequality, 
this means that $\int s^2 d\mu \to + \infty$ among metrics in our arbitrary conformal class $[g]$. In particular,
this shows that there exist metrics on any $4$-manifold  $M$ for which 
$$\int \frac{s^2}{24} d\mu > \int |W_+|^2d\mu$$
and, indeed, that there exist such metrics in any conformal class. 
\end{xpl} 

\begin{xpl} Let  $g$  be a $J$-compatible K\"ahler metric of {\sf non-constant} scalar curvature on a compact complex surface  $(M^4,J)$;
a generic K\"ahler metric in any K\"ahler class will have this property \cite{calabix}. 
By the solution of the Yamabe problem \cite{lp,rick}, there exists a  constant-scalar-curvature metric $\check{g} = u^2 g$ conformal to 
$g$ that   minimizes the Yamabe functional $\mathscr{E}$ in the conformal class $[g]$; and, because $M$ has real dimension $4$, such a 
{\sf Yamabe metric} $\check{g}$  also  minimizes $\int s^2 d\mu$ in its conformal class 
\cite[Proposition 2.1]{bcg1}. Thus $\check{g}$ must satisfy 
$$\int \frac{s^2}{24} d\mu < \int |W_+|^2d\mu$$
because equality is already achieved   by the higher-energy  metric $g$. Thus, any compact  $4$-manifold 
that admits a complex structure of K\"ahler type will admit metrics for which self-dual Weyl outweighs the scalar curvature. 
\end{xpl}

Thus, our question of balance only becomes sensible if we   somehow  turn it   into a  conformally invariant question, 
or    else   narrow the scope of  the question in a way  that  effectively precludes  conformal rescaling. One particularly nice such modification, which coincides with the original question in the Einstein case,  is to ask whether 
\begin{equation}
\label{oyster}
 \int_M |W_+|^2 d\mu \stackrel{?}{\geq} \int_M \left(\frac{s^2}{24} - \frac{|\mathring{r}|^2}{2}\right) d\mu ,
\end{equation}
since the Gauss-Bonnet formula \eqref{gb} implies that the right-hand side is also conformally invariant. 
By combining Gauss-Bonnet with the Thom-Hirzebruch signature formula \eqref{th}, it is now easy to see that 
this modified question is exactly equivalent  to asking when
\begin{equation}
\label{shucks}
\frac{1}{4\pi^2}  \int_M |W_+|^2 d\mu \stackrel{?}{\geq} \frac{1}{3} (2\chi + 3\tau ) (M). 
\end{equation}
Asking whether such an inequality holds for {\em all} metrics on a given $M$ is then a question  about the   infimum of the {\em Weyl functional} 
 $$\mathscr{W}([g ]) :=  \int_M \left(|W_+|^2 + |W_-|^2 \right) d\mu_g , $$
which measures the deviation of a conformal class $[g]$  from local conformal flatness. 
Since  equation   \eqref{th}  implies  that 
 $$
 \mathscr{W}([g ]) =  - 12\pi^2 \tau (M) + 2  \int_M |W_+|^2  d\mu_g, 
 $$ knowing the infimum of $\mathscr{W}$ 
 is equivalent  to understanding the differential-topological invariant $\inf_g \int_M |W_+|^2d\mu_g$,
and for our purposes this will  be  the  more  convenient  formulation of  the problem.

The infimum  of the Weyl functional seems to have been first discussed 
by  Atiyah, Hitchin, and Singer \cite{AHS}, who 
discovered that  the infimum is achieved  on $\CP_2$  by  the Fubini-Study metric; indeed, 
they more generally observed that  \eqref{th} implies that any  metric on a compact oriented $4$-manifold $M$ satisfies 
\begin{equation}
\label{avast}
\frac{1}{4\pi^2} \int_M |W_+|^2 d\mu_g \geq 3 \tau (M), 
\end{equation}
 with equality iff $W_-\equiv 0$. This   seems to have then inspired 
 Osamu Kobayashi \cite{okobweyl} to examine the  key example of  
 $M = S^2 \times S^2$, where  inequality  \eqref{avast}   just  becomes the trivial statement that  $\int_M |W_+|^2d\mu \geq 0$, 
but   where 
 this  lower bound is impossible to achieve, since   a theorem of Kuiper \cite{kuiper} guarantees  that 
$S^4$ is the only   simply-connected $4$-manifold that admits a metric with $W_+= W_-\equiv 0$. 
Kobayashi   conjectured that the infimum on $S^2\!\times\! S^2$ is achieved by the 
K\"ahler-Einstein metric  arising  as  the Riemannian product of two  round $2$-spheres of the same radius.
 Kobayashi's evidence for 
this conjecture  was modest, but interesting; by calculating the second variation of $\mathscr{W}$, 
he proved  that this standard Einstein metric is a local minimum of the Weyl functional, and he also  checked  
  that it is the unique global minimizer of the restriction of $\mathscr{W}$  to the  K\"ahler metrics on  $\CP_1 \times \CP_1$. 

While Kobayashi's  evidence was admittedly   fragmentary, 
Matthew  Gursky  later     discovered  a  beautiful   general result \cite{gursky} that  puts  
the question  on an entirely different footing:

\begin{them}[Gursky]
Let $M$ be a compact oriented $4$-manifold such that  $b_+(M)\neq 0$, and let $[g]$ be any conformal class 
with $Y_{[g]} > 0$. Then  
 \begin{eqnarray}
\label{caterpillar} 
\frac{1}{4\pi^2}  \int_M |W_+|^2 d\mu_g {\geq} \frac{1}{3} (2\chi + 3\tau ) (M), 
\end{eqnarray}
with equality iff $[g]$ is the conformal class of  a K\"ahler-Einstein metric. 
\end{them}

\noindent Here we 
recall that the {\em Yamabe constant} $Y_{[g]}$ 
of a  conformal class $[g]$   is by definition the infimum 
of the Yamabe functional $\mathscr{E} (\hat{g})$  over  $\hat{g}\in [g]$, and that $Y_{[g]}$ is positive 
iff $[g]$ contains   a metric of positive scalar curvature. If $M$ is any compact oriented 
$4$-manifold, also recall that  $b_+(M)$ is defined to be  the dimension of any maximal   subspace
of $H^2 (M,  \RR)$ on which the intersection pairing $H^2 (M, \RR) \times H^2(M, \RR)\to \RR$ is positive 
definite;  and since, for   any Riemannian metric $g$ on $M$, the  self-dual/anti-self-dual decomposition 
 $$\Lambda^2 = \Lambda^+\oplus \Lambda^-$$
of the bundle of $2$-forms induces an intersection-form-adapted   decomposition 
$$
\mathcal{H}_g= \mathcal{H}_g^+ \oplus \mathcal{H}_g^-
$$
of the harmonic $2$-forms $\mathcal{H}_g\cong H^2 (M, \RR)$ into eigenspaces of the Hodge star operator 
$\star : \mathcal{H}_g\to \mathcal{H}_g$, it easily 
 follows that  $b_+(M)$ is exactly  the dimension of
the space $\mathcal{H}_g^+$ of  harmonic self-dual $2$-forms on $(M,g)$. Gursky's argument uses  the Weitzenb\"ock formula for
self-dual harmonic $2$-forms, together with a conformal rescaling argument,  to  show that every 
conformal class on any   $M^4$ with $b_+\neq 0$  contains a metric with 
$2\sqrt{6} |W_+| \geq s$ everywhere. Integrating and applying Cauchy-Schwarz, one then 
concludes that 
 \begin{equation}
\label{hookah} 
\int_M |W_+|^2d\mu \geq  \frac{1}{24} \left(Y_{[g]}\right)^2 
\end{equation}
whenever  $Y_{[g]} \geq  0$. The result then follows, because evaluation 
of the   right-hand side of \eqref{oyster}  
at a  Yamabe metric demonstrates   that this expression   is less than (or equal to) 
 the right-hand side of \eqref{hookah} 
 for any conformal class. 

\medskip

While Gursky's theorem 
 certainly seems like a huge step in the direction of answering Kobayashi's question, 
Gursky's   method  unfortunately cannot provide any  information at all about conformal classes with $Y_{[g]} < 0$; and, 
for better or worse, 
 ``most'' conformal classes on any $4$-manifold   inevitably  have negative Yamabe constant. 
Fortunately, an entirely different method does allow us to plunge into this Yamabe-negative  realm; but this method only works
on the  small 
 class of $4$-manifolds with   $2\chi + 3\tau > 0$ that 
admit both symplectic structures and Riemannian metrics with $s > 0$. These are 
exactly   \cite{chenlebweb,lebcam}  the  previously-mentioned  manifolds
 that 
  carry  both a symplectic structure  and a $\lambda > 0$ Einstein metric. Equivalently, 
they are the underlying smooth $4$-manifolds of the  {\em del Pezzo surfaces},
meaning the  compact complex surfaces that have  ample anti-canonical line bundle $K^{-1}$. Up to oriented diffeomorphism, there are exactly
ten  of these manifolds, namely $S^2\! \times\! S^2$ and the connected sums $\CP_2 \# k \overline{\CP}_2$, $k = 0, 1, \ldots , 8$. 
Most of these actually carry  K\"ahler-Einstein metrics, and, in the spirit of Kobayashi's conjecture, one could hope
that their conformal classes  might  exactly minimize  the Weyl functional. 

Because  each del Pezzo $4$-manifold $M$ has  $b_+(M)=1$,  there is, up to a multiplicative constant, a unique 
self-dual harmonic $2$-form $\omega$ on $M$ 
for each Riemannian metric $g$, and this $\omega$ moreover only depends on the conformal class $[g]$ of the given metric. 
When this $\omega$ is everywhere non-zero, it is automatically  a symplectic form, and one therefore says that 
$[g]$ is a conformal class of {\em symplectic type}.  Like Gursky's  condition $Y_{[g]} > 0$, this new condition
 is open in the $C^2$ topology; however, it is also genuinely different, because 
 one can  construct sequences of conformal classes  $[g_j]$
 of symplectic type with $Y_{[g_j]} \to -\infty$. Nonetheless, inequality \eqref{caterpillar} can still  be shown \cite{lebdelpezzo} 
 to hold in  these Yamabe-negative depths: 

  \begin{them}[L]
 Let $M^4$ be the  underlying  smooth compact oriented manifold 
 of a del Pezzo surface. Then  any conformal class $[g]$ of {\em symplectic type}
 on $M$  satisfies  
$$ \frac{1}{4\pi^2}  \int_M |W_+|^2 d\mu_g \geq \frac{1}{3} (2\chi + 3\tau ) (M), $$
 with equality iff $[g]$ contains a K\"ahler-Einstein metric $g$ with $s >0$. 
  \end{them}

  \noindent Here the method of proof focuses on choosing a representative $g$ for the conformal class $[g]$ 
  for which  the harmonic $2$-form $\omega$ has  pointwise norm $|\omega |\equiv \sqrt{2}$;
  this makes  $(M, g, \omega )$ into an {\em almost-K\"ahler} manifold with $c_1\cdot [\omega ]> 0$,
  and  this then gives rise to sharp lower bounds for the Weyl functional. It is also worth mentioning that there 
 are two del Pezzo manifolds, $\CP_2\#  \overline{\CP}_2$ and $\CP_2\# 2 \overline{\CP}_2$, where the associated 
 Einstein metric is only {\em conformally} K\"ahler, but nevertheless still appears to minimize the Weyl functional  \cite{lebdelpezzo}.

  
  The upshot is that Kobayashi's conjecture seems increasingly plausible for  $S^2\! \times\! S^2$
  and its del Pezzo cousins. But could  the lower bound \eqref{caterpillar} also hold
  in the Yamabe-negative realm on many other manifolds? One might first worry about Taubes'  theorem \cite{tasd}, asserting 
  that, for any compact oriented $4$-manifold $M$, the connected sum $M\# k \overline{\CP}_2$ will admit   metrics with $W_+\equiv 0$ 
  for  astronomically large  $k \gg 0$; but $2\chi + 3\tau \ll 0$ for these examples, so    \eqref{caterpillar} becomes  a tautology in this context. 
  More pertinently, 
  compact hyperbolic $4$-manifolds
   have $W_+\equiv 0$ and $2\chi + 3\tau > 0$, and so certainly violate
  \eqref{caterpillar}; but such  locally-symmetric examples  never admit positive-scalar-curvature metrics \cite{gvln2},
  and thus  do not  carry  any metrics to  which Gursky's result  applies.  This makes it  worth asking 
  whether there are $4$-manifolds   that {\em do} carry positive-scalar-curvature metrics, 
  but  which also carry some Yamabe-negative conformal classes for which   \eqref{caterpillar} fails. Our first  main result 
 is that this phenomenon is  actually extremely common:

\begin{main}
\label{madteaparty} For any sufficiently large integer $m$, the smooth compact simply-connected  spin manifold 
$$M= m(S^2\!\times\! S^2):= \underbrace{(S^2\!\times\! S^2)\# \cdots \# (S^2\!\times\! S^2)}_m$$
admits Riemannian conformal classes  $[h]$ such that 
\begin{equation}
\label{rabbit} 
\frac{1}{4\pi^2} \int_{M} |W_+|^2 d\mu_{h} < \frac{1}{3} (2\chi + 3\tau )(M).
\end{equation}
Similarly, for any any sufficiently large integer $m$ and 
any integer $n$ such that  $\frac{n}{m}$ is sufficiently close to 
$1$, the smooth compact  simply-connected  non-spin manifold 
$$M= m\CP_2 \# n \overline{\CP}_2 := 
 \underbrace{\CP_2 \# \cdots \# \CP_2}_m\#  \underbrace{ \overline{\CP}_2 \# \cdots \#  \overline{\CP}_2}_n$$ 
admits conformal classes  $[h]$  that satisfy inequality \eqref{rabbit}. 
\end{main}

\noindent   The proof of this result can be found in \S  \ref{worksite}  below. 

\medskip 

On the other hand, given the role of almost-K\"ahler geometry in the above discussion, it also seems natural to explore 
our question of balance in the almost-K\"ahler context. Here, the scales can tip either way. Indeed,  if we  choose to impose 
one  of the  two key  conditions that 
played a role in our previous discussion, systematic but opposing   patterns emerge:

\begin{main}
\label{knight} 

 If  $(M,g, \omega)$ is a  compact almost-K\"ahler $4$-manifold  such that $\delta W_+=0$, where $\delta$ denotes 
the divergence operator,  then 
$$
\int_M \frac{s^2}{24}d\mu_g \geq  \int_M |W_+|^2d\mu_g ,
$$
with equality iff $(M, g ,\omega )$ is  K\"ahler.  By contrast, if   $(M,g, \omega)$
instead has scalar curvature $s\geq 0$, then 
$$
\int_M |W_+|^2d\mu_g  \geq
 \int_M \frac{s^2}{24}d\mu_g   ,
$$
again with equality iff $(M, g ,\omega )$ is  K\"ahler. In particular, 
any compact almost-K\"ahler $4$-manifold  $(M,g, \omega)$ with 
 $\delta W_+=0$ and  $s\geq 0$ is necessarily K\"ahler. 
\end{main} 

\noindent For the proof, see   \S \ref{hummingbird} below.

\section{Curvature and Connected Sums}
\label{worksite}

The constructions in this section will  depend   on the existence of {\sf simply-connected}
minimal complex surfaces $(X^4, J)$  of general type with  $\tau (X) > 0$, 
where  the signature 
$\tau (X) = b_+(X) - b_-(X)$ 
  is understood to be computed  with respect to the  complex orientation of $X^4$. 
Recall \cite{bpv} that a complex surface $(X,J)$  is said to be {\sf minimal}
if it is contains no holomorphically embedded $\CP_1$ of homological self-intersection $-1$, and that $(X^4,J)$ 
is said to be of {\sf general type} if $h^0(X, \mathcal{O} (K^{\otimes j}))$ grows quadratically in 
$j$ for $j\gg 0$,  where $K=\Lambda^{2,0}_X$ denotes the canonical line bundle of $X$. Minimal, simply-connected
complex surfaces of general type exist in abundance; indeed, every smooth complete-intersection surface 
in $\CP_n$ of degree  $\geq 9$ 
 is an example. However,
  constructing such complex surfaces  with  $\tau (X) > 0$ is  surprisingly  difficult  and 
subtle, and no examples were known prior to the trailblazing work of   Miyaoka \cite{miyaokaplus} and Moishezon-Teicher \cite{moteplus}. 
A plethora  of {\sf non-spin} examples   were then   constructed  by Chen \cite{zjchenplus}, after which 
 Persson, Peters, and Xiao \cite{perssonspinplus}  proceeded to show 
 that {\sf spin} examples  exist in similar profusion.
Much more recently, Roulleau and Urz\'{u}a \cite{ruler} settled a celebrated  problem in complex-surface geography
by showing that there  exist sequences of such $X$ with $c_1^2(X)/c_2(X) \to 3$; moreover, one can either do this  while  insisting 
that these $4$-manifolds $X$ be {\sf spin}, or while   instead insisting that that they be {\sf non-spin}. In  terms 
of the {\sf signature} $\tau= (c_1^2 - 2c_2)/3$ and {\sf topological Euler characteristic} $\chi  = c_2$, 
Roulleau and Urz\'{u}a's construction yields  
 sequences of simply-connected  complex surfaces with  $\tau (X) /\chi (X) \to \nicefrac{1}{3}$.

 These complex surfaces $X$ will eventually become essential building blocks in our construction. To make good use of them, however, 
 we will first need to introduce some basic  differential-geometric tricks. 
 
 \begin{lem} 
 \label{fork} 
 Let $\epsilon > 0$ be given. Then, for any smooth 
  compact oriented Riemannian $4$-manifold $(Y,g_0)$, there is a  Riemannian metric $g_\epsilon$
 on $Y$ which is flat on some tiny ball, but which also satisfies 
 \begin{equation}
\label{closer} 
 \frac{1}{4\pi^2} \int_{Y} |W_+|^2 d\mu_{g_\epsilon} <  \frac{1}{4\pi^2} \int_{Y} |W_+|^2 d\mu_{g_0} + \epsilon . 
\end{equation}
Similarly, if  $(Y^4,g_0)$ is a compact Riemannian orbifold with only isolated singularities, there exists
an orbifold metric $g_\epsilon$ on $Y$ which is flat in a small neighborhood of each orbifold singularity, 
but  also satisfies  \eqref{closer}.  
 \end{lem} 
 \begin{proof}
 In geodesic normal coordinates about some point $p\in Y$, $$g_0= \delta + O(\varrho^2),
$$
where $\delta$ is the flat Euclidean metric associated with the coordinate system, and $\varrho$
is the Euclidean radius. Let
$\phi: {\mathbb R}\to {\mathbb R}$ be a non-negative smooth function 
which is identically $0$ on $(-\infty , \frac{1}{2})$ and 
identically $1$ on $(1,\infty ) $, and, for each sufficiently small 
$t> 0$,
set
$$h_t=\delta  + \phi (\frac{\varrho}{t}) [g_0 - \delta] ,$$ 
so that $h_t$  coincides with $g_0$ for  $\varrho > t$, but is flat for $\varrho < t/2$.
We now extend this to a metric on all of $Y$ by setting $h_t = g_0$ outside our coordinate chart. 
In the transition region $\varrho \in (t/2 , t)$, for any small $t$, 
one then has 
$$
 \|h_t-\delta\| \leq C t^{2}, \qquad   \| {\mathbb D}h_t\| \leq C t , \qquad  \| {\mathbb D}^{2}h_t\| \leq C, 
$$
where ${\mathbb D}$  is the  Euclidean derivative operator   associated
with the given coordinate system,  and where  the constant $C$ is
independent
of $t$. In this transition annulus, 
the norm-square $|\mathcal{R}|^2$ of the curvature tensor of $h_t$  is therefore everywhere less than a constant $C^\prime$ 
independent of $t$, 
and the same is therefore true of $|W_+|^2 \leq |\mathcal{R}|^2$. Since 
the volume of this transition annulus is less than a constant times $t^4$, 
 the effect of this modification  on  $\|W_+\|^2_{L^2}$ is less than a constant times
$t^4$. We can therefore achieve our goal by setting $g_\epsilon = h_t$ for some sufficiently small $t$. 

In the orbifold case, the proof is the same, except that, instead of altering the given metric on a single ball, 
 we instead change it  in the above manner
on a  finite number of neighborhoods modeled on $B^4/\Gamma_j$ for 
appropriate  finite
subgroups  $\Gamma_j  \subset \mathbf{SO}(4)$. 
 \end{proof}

 \begin{lem}
 \label{spoon} 
 Let $(Y_1, g_1)$ and $(Y_2, g_2)$ be two compact oriented $4$-manifolds, where each one is conformally flat on 
 some small open set. Then the connected sum $Y_1 \# Y_2$ admits a conformal class of metrics $[g]$ such that 
 \begin{equation}
\label{additive} \int_{Y_1\# Y_2} |W_+|^2 d\mu_{g} =  \int_{Y_1} |W_+|^2 d\mu_{g_1} + \int_{Y_2} |W_+|^2 d\mu_{g_2}.
\end{equation}
 Similarly, if $(Y_1, g_1)$ and  $(Y_2, g_2)$ are oriented orbifolds which contain flat neighborhood
 modeled on $B^4/\Gamma$ and  $\overline{B^4/\Gamma}$, respectively, where the bar is used to  indicate
reversal  of  orientation, then the generalized connected sum $Y_1 \#_{\Gamma} Y_2$ admits
a conformal class of orbifold metrics $[g]$ for which  $\| W_+\|^2_{L^2}$ is exactly additive, in the sense 
 that  a perfect analog of     \eqref{additive} holds. 
\end{lem} 
\begin{proof} We may delete a tiny round ball from the conformally flat region of each manifold, and  then form
the connected sum by identifying conformally-flat  annular regions near the boundary spheres by an orientation-preserving 
a conformal inversion. Since these
gluing maps have been chosen to preserve the conformal structure, there is then a well-defined  conformal structure induced on the connected sum. 
Since $\| W_+\|^2_{L^2}$ is conformally invariant, the conformally flat balls we have deleted have no impact whatsoever on the
curvature integral, and integrals therefore coincide with the sum of the integrals in the ball-complements. 
This   moreover works equally well in the orbifold case, with only cosmetic changes. 
\end{proof} 

With these lemmata in hand, we can now prove the following: 

\begin{prop}
\label{key}
 Let $(X,J)$ be a minimal complex surface of general type, and let $\epsilon > 0$ be given. 
Then the smooth compact oriented $4$-manifold $X$ admits 
a conformal class $[g_\epsilon ]$  of Riemannian metrics such that
\begin{equation}
\label{bump} 
\frac{1}{4\pi^2} \int_{Y} |W_+|^2 d\mu_{g_\epsilon} < \frac{1}{3} c_1^2 (X) + \epsilon . 
\end{equation}
\end{prop} 
\begin{proof}
Let $\mathscr{X}$ be the pluricanonical model of $(X,J)$, which is obtained by collapsing all $(-2)$-curves
in $X$. Since $\mathscr{X}$ has only rational double-point singularities, we may choose to view it as a complex orbifold that has 
only A-D-E singularities. This orbifold then has $c_1 < 0$, and the usual  Aubin-Yau  proof \cite{yau}  therefore implies 
 \cite{tsuj}  that it caries an orbifold K\"ahler-Einstein metric $\check{g}$. This K\"ahler-Einstein metric then satisfies
 $$
 \frac{1}{4\pi^2} \int_{\mathscr{X}} |W_+|^2 d\mu_{\check{g}} =  \frac{1}{4\pi^2} \int_{\mathscr{X}} \frac{s^2}{24}  d\mu_{\check{g}}
 = \frac{1}{3} c_1^2( \mathscr{X}) = \frac{1}{3} c_1^2( X),
 $$
 where the last equality reflects the fact that $X\to \mathscr{X}$ is a crepant resolution. If $\mathscr{X}=X$, we are already done. 
 Otherwise, for each singular point $p_j$ with orbifold group $\Gamma_j\subset \mathbf{SU}(2)$, let  $(Y_j, [g_j]) $ be  the one-point conformal compactification
 of one of  Kronheimer's gravitational instanton metrics \cite{kronquot} on the minimal resolution of 
 $\CC^2/\Gamma_j$; this  is a smooth compact orbifold
 with $W_+\equiv 0$ that has  a single, isolated singularity modeled on $\overline{B^4/\Gamma_j}$. 
  Flattening these
 orbifolds slightly at their singular points, in accordance with Lemma \ref{fork}, and then performing generalized connected sums, 
 in accordance with Lemma \ref{spoon}, we then obtain a conformal metric on 
 $X = \mathscr{X}\#_{\Gamma_1} Y_1 \#_{\Gamma_2}  \cdots \#_{\Gamma_k} Y_k$
 with  $\|W_+\|^2_{L^2}$ as close as we like to its value  for the orbifold 
$(\mathscr{X}, [\check{g}])$. 
\end{proof}

This now puts us in a position to prove one of our key results:

\begin{thm}  \label{marchhare}
For any sufficiently large integer $m$, the smooth  compact simply-connected spin $4$-manifold 
$M= m(S^2\!\times\! S^2)$
 admits conformal classes $[h]$
that satisfy inequality \eqref{rabbit}. \end{thm}

\begin{proof} There are \cite{perssonspinplus,ruler} infinitely many simply-connected  compact complex surfaces 
$X$ of general type  with   signature $\tau (X) > 0$ that are {\sf spin} (and so, in particular,  minimal). Choose such a complex surface  $X$, and,
for a small $\epsilon$ we will specify later, 
 equip $X$, per  Lemma \ref{fork} and Proposition \ref{key}, with a   metric $g_\epsilon$ 
that  satisfies \eqref{bump} and is flat on four tiny balls. We also equip the orientation-reversed version $\overline{X}$
of $X$ with  mirror-image versions of this  $g_\epsilon$, and finally equip $S^2 \times S^2$ with  a modification
$\mathbf{g}_\epsilon$ of its  standard product K\"ahler-Einstein metric that  is  flat on a tiny ball and has 
$$
\frac{1}{4\pi^2} \int_{S^2\! \times\! S^2} |W_+|^2 d\mu_{\mathsf{g}_\epsilon} < \frac{c_1^2}{3} (\CP_1\times \CP_1)  + \epsilon = \frac{8}{3} +\epsilon.
$$ 
By removing round, conformally  flat balls and  gluing as in Lemma \ref{spoon}, we thus  obtain a conformal class  $[ h_\epsilon ]$ on 
$$
M_{k,\ell} := (k+\ell ) [ X\# \overline{X}] \# (2k +\ell) [S^2 \times S^2]
$$
that satisfies 
\begin{eqnarray*}
\frac{1}{4\pi^2} \int_{M_{k,\ell}} |W_+|^2 d\mu_{h_\epsilon} &= &   \frac{k+\ell}{4\pi^2} \left[ \int_X |W_+|^2 d\mu_{g_\epsilon} 
 +  \int_{\overline{X}}
 |W_+|^2 d\mu_{g_\epsilon} \right] \\&& \hspace{1.5in} 
+ \frac{2k+\ell}{4\pi^2} \int_{S^2 \times S^2} |W_+|^2 d\mu_{\mathsf{g}_\epsilon} \\
 &=&   \frac{k+\ell}{4\pi^2}\left[  2 \int_X |W_+|^2 d\mu_{g_\epsilon} - 12\pi^2 \tau (X)   \right]  \\&& \hspace{1.5in} 
+ \frac{2k+\ell}{4\pi^2} \int_{S^2 \times S^2} |W_+|^2 d\mu_{\mathsf{g}_\epsilon} \\ \\
  &< &  (k+\ell)  \left[ \frac{2}{3} (2\chi + 3\tau )  (X) - 3\tau (X) + 2\epsilon \right]   + \frac{8}{3}(2k+\ell)  + (2k+\ell) \epsilon \\
  &\leq & (k+\ell) \left[ \frac{4}{3} \chi (X) - \tau (X) + \frac{16}{3} + 4\epsilon \right].
\end{eqnarray*}
By contrast, 
\begin{eqnarray*}
\frac{1}{3} (2\chi + 3\tau ) (M_{k,\ell}) &=& \frac{2}{ 3} \chi (M_{k,\ell}) = \frac{2}{3}~ [ 2+   b_2(M_{k,\ell})] \\
&=& \frac{4}{3}  + \frac{2}{3}(k+\ell) [ b_2 (X) + b_2 (\overline{X})]  + \frac{4}{3} (2k+\ell)  \\
&=& \frac{4}{3}  + \frac{4}{3}(k+\ell) b_2 (X) +  \frac{4}{3} (2k+\ell) \\
&>  &    \frac{4}{3}(k+\ell) [ b_2(X) +1 ] \\
&=& (k+\ell)  \left[ \frac{4}{3} \chi (X) -  \frac{4}{3}\right] . 
\end{eqnarray*}
Taking  $\epsilon \in (0, \frac{1}{12})$, we thus deduce that 
\begin{equation}
\label{walrus} 
\frac{1}{3} (2\chi + 3\tau ) (M_{k,\ell}) - \frac{1}{4\pi^2} \int_{M_{k,\ell}} |W_+|^2 d\mu_{h_\epsilon} > (k+\ell) \left[\tau (X) - 7 \right].
\end{equation}
 But since $X$ is a spin manifold, Rokhlin's Theorem \cite{lawmic,rokhlin}  tells us that $16 |  \tau (X)$,  so  our  $\tau (X) > 0$ hypothesis 
therefore    implies that $\tau (X) \geq 16$. Thus,  the right-hand-side of \eqref{walrus} is  automatically positive, and we have therefore produced   conformal classes $[h]$ 
  on ${M_{k,\ell}}$ that satisfy  \eqref{rabbit}. 
 
 On the other hand, 
 since $X\# \overline{X}$ is a simply connected spin manifold of signature zero, Wall's stable classification \cite{wall} via $h$-cobordism implies that
 there exists some large integer $\mathfrak{p}$ such that 
 $X\# \overline{X} \# (k+\ell) (S^2 \times S^2)$ is diffeomorphic to a connected sum $\mathfrak{q} (S^2 \times S^2)$ for any  $(k+\ell )\geq \mathfrak{p}$. By 
 induction on the number of   $X\# \overline{X}$ summands,   and  then  adding 
 $k$ additional    $(S^2 \times S^2)$         summands, we thus see  that 
  ${M_{k,\ell}} = (k+\ell) ( X\# \overline{X}) \# (2k+\ell) (S^2 \times S^2)$ is   diffeomorphic to $m (S^2 \times S^2)$
  whenever $(k+\ell )\geq \mathfrak{p}$, 
  where  we have set $m := k[b_2(X)+2] + \ell [b_2(X) + 1]$. 
  
On the other hand,  any integer $m \geq b_2(X) [b_2(X)+1]$ can be expressed 
  as $k[b_2(X)+2] + \ell [b_2(X) + 1]$ for some integers $k, \ell \geq 0$;  this elementary fact 
 is  actually a special case of   Sylvester's solution \cite{sylvester} of the Frobenius two-coin problem. But since 
 this  expression for $m$ implies that  $(k+\ell) [b_2(X)+2]\geq m$, we also  have  $(k+\ell )\geq \mathfrak{p}$ whenever 
  $m \geq \chi (X) \mathfrak{p}$. Thus, whenever $m$ exceeds $\chi (X) \max ( \mathfrak{p}(X), \chi (X))$, the connected sum
   $m (S^2 \times S^2)$
can also be expressed as   ${M_{k,\ell}}$ for some $k$ and $\ell$, and  consequently 
  admits   a conformal class 
 $[h]$ that satisfies \eqref{rabbit}.   \end{proof} 
 
 Of course,   by the Gromov-Lawson 
surgery theorem   \cite[Theorem A]{gvln},  
the smooth $4$-manifolds  $m(S^2\!\times\! S^2)$ all admit metrics of positive scalar curvature. 
and Gursky's inequality \eqref{caterpillar} then gives an interesting  
lower bound on the Weyl functional on these Yamabe-positive conformal classes. The point of Theorem \ref{marchhare}, however, 
is that   the infimum of the Weyl functional is usually   considerably lower than 
one might guess without taking a plunge into the Yamabe-negative depths. 

Essentially the  same phenomenon also occurs on non-spin $4$-manifolds: 

\begin{thm} \label{madhatter} Choose any $\varepsilon \in (0, \frac{1}{5})$. Then, for 
 every sufficiently large  integer $m$, and for  any    integer  $n$  satisfying $(\frac{4}{5} + \varepsilon) m < n <  ( 2- \varepsilon) m$, 
the smooth  compact simply-connected non-spin $4$-manifold
 $M= m\CP_2 \# n \overline{\CP}_2$ 
 admits conformal classes $[h]$
that satisfy inequality \eqref{rabbit}. 
\end{thm}
\begin{proof} By imitating  the proof of Theorem \ref{marchhare}, we first handle the case where $n=m$. Thus,  we  begin by   considering 
$$
M_{k,\ell} = (k+\ell ) [ X\# \overline{X}] \# (2k +\ell) [S^2 \times S^2],
$$
but now take $X$ to be a simply connected {\sf non-spin} minimal  complex surface of general type  with signature  $\tau (X) \geq 8$. 
In this setting,   Wall's stable classification \cite{wall} implies  that there is some $\mathfrak{p}$ such that the simply-connected zero-signature
 non-spin $4$-manifold $M_{k,\ell}$ is diffeomorphic to $m\CP_2 \# m \overline{\CP}_2$ whenever $(k+\ell ) \geq \mathfrak{p}$, where 
 $m := k\chi (X)  + \ell [\chi (X) -1]$.  
 Since any integer  $m \geq [\chi (X)]^2 $ can be expressed as  $k\chi (X)  + \ell [\chi (X) -1]$  for   integers $k , \ell \geq 0$, and since these 
  integers will then satisfy    $(k+\ell ) \geq m /\chi (X)$, it in particular follows that 
  $(k+\ell ) \geq \mathfrak{p}$ whenever $m \geq \chi(X) \mathfrak{p}$. 
  Hence  $m \CP_2 \# m \overline{\CP}_2$ is diffeomorphic to some $M_{k,\ell}$  whenever $m \geq \chi (X) \max (\mathfrak{p}, \chi (X))$, 
   On the other hand, 
 our previous  gluing construction now yields  conformal metrics $[h]$ on $M_{k,\ell}$ which satisfy 
 $$
 \frac{1}{3} (2\chi + 3\tau ) (M_{k,\ell})  - \frac{1}{4\pi^2} \int_{M_{k,\ell}} |W_+|^2 d\mu_{h} > (k+\ell ) \left [\tau (X) - 7\right] \geq m \frac{\tau (X)-7}{\chi (X)} .
 $$ 
 Thus, whenever  $m$ is large,  $m \CP_2 \# m \overline{\CP}_2\approx M_{k,\ell}$ admits  
a  conformal class that  not only satisfies
 \eqref{rabbit}, but for which we actually  have a lower bound for  the gap in terms of $m$ and the homeotype of our  chosen building-block $X$. 
 
We next consider the manifolds 
$$\widehat{M}_{j,k,\ell} = M_{k,\ell} \# j \overline{\CP}_2,$$
and notice that $m \CP_2 \# (m+j)\overline{\CP}_2$ is then  diffeomorphic to 
some such $\widehat{M}_{j,k,\ell}$ whenever $m$ is sufficiently large. But since  the mirror-image 
Fubini-Study metric on $\overline{\CP}_2$ has $W_+\equiv 0$,  Lemma \ref{fork} guarantees that this
reverse-oriented version of $\CP_2$ 
carries conformal classes with $\frac{1}{4\pi^2} \int |W_+|^2 d\mu < \epsilon$ that are conformally flat on some 
tiny ball. Since the constructed conformal classes $[h]$ on $M_{k,\ell}$ also contain tiny conformally-flat
regions, gluing per  Lemma \ref{spoon} thus produces   conformal
classes on $\widehat{M}_{j,k,\ell}$ with 
$$\frac{1}{4\pi^2} \int_{\widehat{M}_{j,k,\ell}} |W_+|^2 d\mu < \frac{1}{4\pi^2} \int_{{M}_{k,\ell}} |W_+|^2 d\mu  + j \epsilon$$
for $\epsilon$ as small as we like. On the other hand, 
$$\frac{1}{3} (2\chi + 3\tau ) (\widehat{M}_{j,k,\ell}) = \frac{1}{3} (2\chi + 3\tau ) (M_{k,\ell}) -\frac{j}{3}$$
so that 
$$
\frac{1}{3} (2\chi + 3\tau ) (\widehat{M}_{j,k,\ell})- \frac{1}{4\pi^2} \int_{\widehat{M}_{j,k,\ell}} |W_+|^2 d\mu > 
m \frac{\tau (X) -7}{\chi (X)} - \frac{j}{3} - j \epsilon
$$
and, by taking $\epsilon$ sufficiently small,  our construction therefore produces  conformal classes $[h]$ on $\widehat{M}_{j,k,\ell}$
satisfying \eqref{rabbit} whenever 
$$0\leq j < 3m \frac{\tau (X) -7}{\chi (X)}.$$
Setting $n= m+j$, our construction 
 therefore  yields conformal classes on $m \CP_2 \# n \overline{\CP}_2$
that satisfy  \eqref{rabbit}, provided that   $m \geq \chi (X) \max (\mathfrak{p}, \chi (X))$ and 
\begin{equation}
\label{southern}
m\leq n <\left(1 + 3 \frac{\tau (X) -7}{\chi (X)}\right) m.
\end{equation}

We can similarly  construct controlled conformal classes  on 
$$
\widecheck{M}_{j,k,\ell} = j \CP_2 \# M_{k,\ell}
$$ 
by instead conformally gluing in  our mild modifications  of the Fubini-Study metric   in a way 
that is compatible with the standard orientation of $\CP_2$. While each copy of $\CP_2$ now contributes 
a substantial additional amount additional  self-dual Weyl curvature, we still have 
$$\frac{1}{4\pi^2} \int_{\widecheck{M}_{j,k,\ell}} |W_+|^2 d\mu < \frac{1}{4\pi^2} \int_{{M}_{k,\ell}} |W_+|^2 d\mu  + 3j + j\epsilon .$$
This is  mitigated by the fact that each added  $\CP_2$  also increases $2\chi + 3\tau$: 
$$\frac{1}{3} (2\chi + 3\tau ) (\widecheck{M}_{j,k,\ell}) = \frac{1}{3} (2\chi + 3\tau ) (M_{k,\ell}) +\frac{5}{3}j. $$
With the  slight change of notation of  now setting 
$n = k \chi (X) + \ell [ \chi (X) -1]$, we  thus have
$$
\frac{1}{3} (2\chi + 3\tau ) (\widecheck{M}_{j,k,\ell})- \frac{1}{4\pi^2} \int_{\widecheck{M}_{j,k,\ell}} |W_+|^2 d\mu > 
n \frac{\tau (X) -7}{\chi (X)} - \frac{4}{3}j  - j \epsilon
$$
so that, for sufficiently small $\epsilon$, our constructed conformal classes satisfy \eqref{rabbit} 
whenever 
$$ 0\leq j < n \frac{3}{4}  \frac{\tau (X) -7}{\chi (X)}.
$$
Setting $m = n+ j$,  we  have thus constructed conformal classes on $\widecheck{M}_{j,k,\ell}$
that satisfy \eqref{rabbit} whenever 
\begin{equation}
\label{northern} 
\frac{m}{1+ \frac{3}{4}  \frac{\tau (X) -7}{\chi (X)}} < n \leq m.
\end{equation}
However, since  $X$  is  of general type,  it 
satisfies \cite{bpv,miyau,yau} the   Miyaoka-Yau inequality  $\chi (X) \geq   3\tau$,  and  
 \eqref{northern}  therefore  implies that   $n > \frac{4}{5} m$.  Whenever inequality  \eqref{northern} holds and $m \geq \frac{5}{4} \chi (X) \max (\mathfrak{p}, \chi (X))$, 
 our previous arguments therefore 
imply that  
$m \CP_2 \# n \overline{\CP}_2$ is  diffeomorphic to  some $\widecheck{M}_{j,k,\ell}$ on which 
 our construction yields conformal  metrics  that satisfy \eqref{rabbit}. 

Up until this point, our discussion has in principle worked for essentially  any non-spin simply connected minimal complex surface  $X$
of positive signature. To optimize
our conclusion, however, we now invoke  the beautiful and surprising theorem  of Roulleau and Urz\'{u}a \cite{ruler}, which asserts 
that there exist sequences of such $X$ such that $c_1^2 (X) /c_2(X) \to 3$. Now notice that  the Euler characteristic
$\chi (X) = c_2(X)$ must tend to infinity for any such  sequence, since the Miyaoka-Yau inequality $c_1^2 \leq 3c_2$ is only saturated  \cite{yau}
by 
ball quotients, which are of course never simply-connected. Consequently, 
$$
\frac{\tau (X) -7}{\chi (X)} = \frac{1}{3} \left[ \frac{c_1^2 (X)}{c_2 (X)} - 2- \frac{21}{c_2 (X)} \right] \to \frac{1}{3}
$$
for any such sequence. Given any $\varepsilon \in (0,\frac{1}{5})$, we can therefore choose such an $X$ such that 
$$
\frac{\tau (X) -7}{\chi (X)} > \frac{1 -\varepsilon}{3} 
$$
and this choice will then satisfy both
$$
1+ 3 \frac{\tau (X) -7}{\chi (X)} > 2- \varepsilon \qquad \mbox{and} \qquad  \frac{1}{1+ \frac{3}{4}  \frac{\tau (X) -7}{\chi (X)}}  < \frac{4}{5} + \varepsilon . 
$$
Thus, if  $m \geq \frac{5}{4}\chi (X) \max (\mathfrak{p} (X), \chi (X))$ and $( \frac{4}{5} + \varepsilon) m < n < (2-\varepsilon) m$, 
either    \eqref{southern} or \eqref{northern} must hold for $m$ and $n$ with  this choice of $X$,  and  $m\CP_2 \# n \overline{\CP}_2$ is consequently  
diffeomorphic to one of the manifolds $\widecheck{X}_{j,k,\ell}$ or $\widehat{X}_{j,k,\ell}$ on which have constructed
a conformal class satisfying inequality  \eqref{rabbit}. 
 \end{proof} 

Together, Theorems  \ref{marchhare} and \ref{madhatter} now imply Theorem \ref{madteaparty}. 

\medskip

Of course, the Gromov-Lawson 
surgery theorem   \cite[Theorem A]{gvln} again implies that the connected sums  $m \CP_2 \# n \overline{\CP}_2$ 
all  admit metrics of positive scalar curvature,  and  Gursky's inequality \eqref{caterpillar} then gives us an interesting 
lower bound for the Weyl functional  of all such metrics. 
On the other hand, Theorem \ref{madhatter}  produces metrics which violate this inequality, thus showing
that the infimum of the Weyl functional is actually unexpectedly  small. This is only  possible
because  the constructed conformal classes $[h]$ all have very    negative Yamabe constants $Y_{[h]}$;  
specifically, these conformal classes $[h]$ must necessarily  satisfy 
 \begin{equation}
\label{hoopla}
Y_{[h]} < - 2\sqrt{6} ~ \| W_{+}\|_{L^2,h}.
\end{equation}
Indeed, let $h$ be a Yamabe metric in $[h]$, and notice that \eqref{rabbit} can be re-expressed as 
$$
\frac{3}{4\pi^2} \int_{M} | W_+|^2 d\mu_{h} <  (2\chi + 3\tau )(M) = \frac{1}{4\pi^2} \int_M \left( \frac{s^2}{24}+  2 |W_+|^2 -
\frac{|\ro |^2}{2} \right) d\mu_h.  
$$
 It  therefore follows that 
 $$
 \int_M  \frac{s^2}{24} d\mu_h > \int_{M} | W_+|^2 d\mu_{h}, 
$$
and  since $h$ has been chosen to be a Yamabe metric, this means that
\begin{equation}
\label{uppsy} 
[Y_{[h]}]^2 = \left( s_h  \vol_h^{1/2} \right)^2 > 24~ \|W_+\|^2_{L^2,h}.
\end{equation}
But since the constructed conformal classes  live on $4$-manifolds with $b_+\neq 0$, 
Gursky's theorem  \cite[Theorem 1]{gursky} assures  us that  \eqref{rabbit}
can only happen for a conformal class $[h]$ with $Y_{[h]} < 0$. Thus, in the present  context, 
\eqref{uppsy} automatically forces inequality  \eqref{hoopla} to hold.

\section{Goldberg Variations}
\label{hummingbird} 

We now wrap up  our exploration   of  the balance between  scalar curvature and self-dual Weyl curvature  by examining   those 
{\sf almost-K\"ahler} $4$-manifolds that have {\sf harmonic self-dual Weyl curvature}, in the sense that 
\begin{equation}
\label{cocoa} 
\delta W_+:= -\nabla \cdot W_+  =0.
\end{equation}
Since  any $4$-dimensional Einstein manifold satisfies \eqref{cocoa} by the second Bianchi identity, 
 this question offers a possible source of insight into the so-called {\em Goldberg conjecture} \cite{goldberg}, 
 which claims  that any compact almost-K\"ahler Einstein manifold should  be K\"ahler-Einstein.


We begin by observing that any $4$-dimensional almost-K\"ahler manifold $(M,g,\omega)$ satisfies 
$$\Lambda^+\otimes \CC = \CC\omega \oplus K \oplus \overline{K},$$
where 
$K= \Lambda^{2,0}_J$ is the canonical line bundle of the almost-complex structure ${J_a}^b= \omega_{ac}g^{bc}$ on $M$.
Locally choosing a unit section $\varphi$ of $K$, we thus have 
$$\nabla \omega = \theta \otimes \varphi + \bar{\theta} \otimes \bar{\varphi}$$
for a unique $1$-form $\theta \in \Lambda^{1,0}_J$, since  $\omega \perp \nabla\omega$,  and $\nabla \wedge \omega =0$. If 
$$\circledast: \Lambda^+\times \Lambda^+\to \odot^2_0\Lambda^+$$
  denotes the symmetric trace-free product, we therefore have 
$$(\nabla_e \omega ) \circledast  (\nabla^e\omega )= 2|\theta |^2 \varphi \circledast\bar{\varphi}  = -\frac{1}{4} |\nabla \omega|^2\omega \circledast \omega$$
and we thus deduce that 
\begin{eqnarray*} 
\langle W_+ , \nabla^*\nabla (\omega\otimes \omega )\rangle
& = &2W_+(\omega , \nabla^*\nabla \omega  ) - 2W_+(\nabla_e \omega , \nabla^e \omega ) \\
& = & 2 W_+(\omega , \nabla^*\nabla \omega  )  + \frac{1}{2}|\nabla \omega |^2 W_+(\omega , \omega )\\
& = & 2 W_+(\omega , 2W_+ ( \omega ) - \frac{s}{3} \omega  ) + 
\Big[ W_+(\omega , \omega ) -\frac{s}{3}\Big] W_+(\omega , \omega )\\
& = & -\frac{2}{3}s W_+ (\omega , \omega ) + 4 |W_+(\omega )|^2 
+ \Big[ W_+(\omega , \omega ) -\frac{s}{3}\Big] W_+(\omega , \omega )\\
& = &  [W_+(\omega , \omega )]^2 + 4 |W_+(\omega )|^2 - s W_+ (\omega , \omega ),
 \end{eqnarray*}
 where we have used the Weitzenb\"ock formula 
 \begin{equation}
\label{tweedledee} 
0= \nabla^* \nabla \omega - 2 W_+(\omega  ) + \frac{s}{3}\omega
\end{equation}
for the harmonic self-dual $2$-form $\omega$, as well as its consequence 
\begin{equation}
\label{tweedledum}
\frac{1}{2} |\nabla \omega |^2 =  W_+(\omega , \omega ) - \frac{s}{3} ~, 
\end{equation}
arising    from the fact that  $|\omega |^2\equiv 2$.
But if $\delta W_+ = 0$, we also have the Weitzenb\"ock formula
$$
0 = \nabla^*\nabla W_+ + \frac{s}{2} W_+ - 6 W_+\circ W_+ + 2 |W_+|^2 I , 
$$
and for  $M$  compact this therefore implies that 
\begin{eqnarray*} 0&=& 
\int_M \Big\langle\Big( \nabla^*\nabla W_+ + \frac{s}{2} W_+ - 6 W_+\circ W_+ + 2 |W_+|^2 I \Big), \omega \otimes
\omega \Big\rangle d\mu \\
&=& \int_M \Big[ \langle W_+ , \nabla^*\nabla (\omega\otimes \omega )\rangle  + \frac{s}{2} W_+(\omega , \omega ) 
 - 6 |W_+(\omega)|^2+ 2 |W_+|^2 |\omega |^2 \Big]  d\mu  \\
&=&  \int_M \Big[ [W_+(\omega , \omega )]^2    - \frac{s}{2} W_+(\omega , \omega ) 
 - 2 |W_+(\omega)|^2+ 4 |W_+|^2 \Big]  d\mu  . 
 \end{eqnarray*}
Hence any compact almost-K\"ahler $(M^4, g, \omega )$  with $\delta W_+=0$ satisfies
\begin{equation}
\label{siren}
\int s W_+(\omega , \omega ) d\mu = \int  \Big[ 8 |W_+|^2     
 - 4 |W_+(\omega)|^2+ 2[W_+(\omega , \omega )]^2 \Big]  d\mu .
\end{equation}
This has an amusing  application to   our question of  balance:

\medskip 

\begin{thm}
\label{hare} 
 If a  compact almost-K\"ahler $4$-manifold  $(M,g, \omega)$ satisfies $\delta W_+=0$, then 
\begin{equation}
\label{melody}
\int_M \frac{s^2}{24}d\mu_g \geq  \int_M |W_+|^2d\mu_g ,
\end{equation}
with equality iff $(M, g ,\omega )$ is a constant-scalar-curvature K\"ahler manifold. 
\end{thm} 
\begin{proof}
To better understand the meaning  of \eqref{siren}, let us express $W_+$ at an arbitrary   point in an orthonormal basis
$\{ {\mathsf e}_j\}_{j=1}^3$ 
 for 
$\Lambda^+$ in which $\omega = \sqrt{2} {\mathsf e}_1$, and in which $W_+(\omega)$ is orthogonal to ${\mathsf e}_3$. 
Then, in this basis, 
$$W_+ = \left[\begin{array}{ccc}\alpha & \gamma &  \\ \gamma& \beta &  \\ &  &-(\alpha +\beta)\end{array}\right]$$
 for suitable real numbers $\alpha, \beta, \gamma$. In terms of these components, 
$$|W_+|^2 = 2 \alpha^2 + 2\beta^2  + 2\alpha\beta + 2\gamma^2, $$
$$[W_+ (\omega , \omega) ]^2 = 4\alpha^2 ,$$
and 
$$|W_+ (\omega )|^2 = 2\alpha^2 + 2\gamma^2.$$
We therefore have 
\begin{eqnarray*}
4|W_+|^2 - 4|W_+(\omega )|^2 + 2[W_+ (\omega , \omega) ]^2&=&  8\alpha^2+ 8\beta^2  +8\alpha\beta\\
&=&  6\alpha^2 + 8(\frac{\alpha}{2}+ \beta)^2\\
&\geq & 6\alpha^2 = \frac{3}{2} [W_+ (\omega , \omega) ]^2.
\end{eqnarray*}
Thus  \eqref{siren} implies that 
$$
\int s W_+(\omega , \omega ) d\mu \geq  \int  \Big[ 4 |W_+|^2 + \frac{3}{2} [W_+ (\omega , \omega) ]^2
 \Big]  d\mu ,
$$
or in other words that 
$$
\frac{3}{8} \int \left[\frac{2s}{3}  -  W_+(\omega , \omega )\right] W_+(\omega , \omega ) ~d\mu \geq \int |W_+|^2d\mu .
$$
Substituting $W_+(\omega , \omega )= \frac{1}{2} |\nabla \omega|^2 + \frac{s}{3}$ from \eqref{tweedledum}, we thus have 
$$\frac{3}{8} \int \left[\frac{s}{3}  -  \frac{1}{2} |\nabla \omega|^2 \right] \left[\frac{s}{3}  +  \frac{1}{2} |\nabla \omega|^2 \right]  ~d\mu \geq \int |W_+|^2d\mu $$
and algebraic simplification therefore yields 
\begin{equation}
\label{sirensong}
\int_M \frac{s^2}{24} d\mu- \frac{3}{32} \int_M |\nabla \omega|^4 d\mu \geq \int_M   |W_+|^2d\mu.
\end{equation}
Hence  any compact almost-K\"ahler manifold  $(M^4, g , \omega )$ with $\delta W_+=0$ must satisfy  \eqref{melody}, 
with equality iff $(M, g ,\omega )$ is K\"ahler. The claim  therefore 
follows, because   a K\"ahler surface  $(M^4, g , J)$ satisfies
$\delta W_+=0$ if and only if its scalar curvature $s$  is constant. 
\end{proof} 

On the other hand, the balance tips in  the opposite direction  for any almost-K\"ahler manifold of   non-negative
 scalar curvature:

\begin{prop}
\label{tortoise} 
 If  $(M,g, \omega)$ is a   compact almost-K\"ahler $4$-manifold  with scalar curvature $s \geq  0$, then 
$$
\int_M |W_+|^2d\mu_g \geq \int_M \frac{s^2}{24}d\mu_g ,
$$
with equality iff $(M, g ,\omega )$ is a K\"ahler manifold. 
\end{prop} 
\begin{proof} By the Weitzenb\"ock formula \eqref{tweedledee}, we have  
$$2\sqrt{\frac{2}{3}} |W_+| \geq W_+(\omega, \omega ) =  \frac{s}{3} + \frac{1}{2}  |\nabla \omega|^2$$
where the inequality results from  the fact that $W_+$ is trace-free and $|\omega|^2 \equiv 2$. 
Consequently, any almost-K\"ahler $(M^4,g,\omega)$ satisfies 
$$|W_+| \geq \frac{s}{2\sqrt{6}},$$
with equality only   at points where $\nabla\omega =0$. When $s\geq 0$, squaring both sides and integrating thus  yields  the desired result. 
\end{proof}

Theorem \ref{hare} and Proposition \ref{tortoise} now imply Theorem \ref{knight}, along with:

\begin{cor} 
\label{mockturtle}
Any  compact almost-K\"ahler manifold $(M^4, g, \omega)$  with $s \geq 0$   and $\delta W_+ =0$ 
is actually  a constant-scalar-curvature K\"ahler manifold. 
\end{cor}

In the special case where  $g$ is an Einstein metric, this gives a different proof of Sekigawa's partial solution   \cite{seki1} of the $4$-dimensional 
Goldberg conjecture. 
For  related results, see \cite{lebcake}.  

\medskip 

By contrast, however,  if we drop the assumption that $s \geq 0$, there are   explicit examples of compact 
almost-K\"ahler $4$-manifolds with $\delta W_+=0$ that are manifestly   {\sf non-K\"ahler}.
In particular, one can construct  \cite{bisleb,inyoungagag} explicit   compact, 
 strictly almost-K\"ahler manifolds that are {anti-self-dual}; and since   these have $W_+\equiv 0$, they obviously satisfy
  $\delta W_+ =0$, too. Because these anti-self-dual examples  have scalar curvature $s= -\frac{3}{2} |\nabla \omega |^2 \leq 0$,
  with strict inequality at most points, 
they inhabit outlands  that lie well beyond the reach of Corollary \ref{mockturtle}. In particular, these examples show that 
one na\"{\i}ve generalization of the Goldberg conjecture is certainly  false.

\medskip 

Finally, we recall that, in the K\"ahler case, the first Chern class is represented by $\frac{1}{2\pi}\rho$,
where $\rho = r(J \cdot , \cdot)$ is the Ricci form. As a consequence, 
\begin{equation}
\label{ricci} 
c_1^2 (M) = \int_M \frac{\rho}{2\pi}  \wedge  \frac{\rho}{2\pi}  = \frac{1}{8\pi^2}  \int_M \left( \frac{s^2}{4}-|\ro|^2\right) d\mu
\end{equation}
for any K\"ahler manifold of real dimension $4$. However, since 
$$
c_1^2 (M) = (2\chi + 3\tau )(M) = \frac{1}{4\pi^2} \int_M \left( \frac{s^2}{24} + 2|W_+|^2 - \frac{|\mathring{r}|^2}{2}   \right) d\mu_g ,
$$
for any Riemannian metric, 
equation  \eqref{ricci} can instead  be explained by the fact that $|W_+|^2 = \frac{s^2}{24}$ in the K\"ahler case. 
This latter way of understanding \eqref{ricci}   has the advantage of making it  clear that generalizations 
of this formula to other contexts must hinge on our familiar question of balance. For example,
in the almost-K\"ahler context,  Proposition  \ref{tortoise} and Theorem \ref{hare}  immediately imply  the following result:

\begin{cor} Let  $(M, g , \omega )$ be a compact almost-K\"ahler $4$-manifold.
\begin{enumerate}[(i)]
\item 
If $g$ has scalar curvature  $s\geq 0$, then 
$$\frac{1}{8\pi^2}  \int_M \left( \frac{s^2}{4}-|\ro|^2\right) d\mu \leq  c_1^2 (M ),  $$
with equality iff $(M, g , \omega )$  is K\"ahler.  
\item If, instead, $g$ satisfies 
 $\delta W_+=0$, then
$$ \frac{1}{8\pi^2} \int_M \left( \frac{s^2}{4}-|\ro|^2\right) d\mu \geq  c_1^2 (M),  $$
again with equality iff $(M, g , \omega )$  is K\"ahler.  
\end{enumerate}
\end{cor}

This once again illustrates the degree to which the question of balance consistently plays a natural  role
in understanding  the relationship between curvature and the topology of smooth compact
Riemannian $4$-manifolds.

\section{Unanswered Questions}

On  simply connected compact $4$-manifolds that carry metrics of positive scalar curvature, 
we have just seen that  the infimum of the Weyl functional is often substantially smaller than one
might have guessed on the basis of  Gursky's inequality \eqref{caterpillar}. In particular, Theorem  \ref{marchhare} 
asserts that there exists  an integer $m_0$ such that  $m(S^2\!\times\! S^2)$ carries  a conformal class satisfying \eqref{rabbit} whenever 
$m\geq m_0$. But the method of proof used here does not actually display  a concrete 
$m_0$ with this property; nor  does it even hint at what might happen when $m$ is reasonably  small. 
 Thus, while one  might hope that this  phenomenon  already occurs  when, say,  $m=2$, proving or disproving such
a statement might require an entirely different set of ideas. Moreover, the present method only  gives us 
a crude upper bound for the infimum of the Weyl functional, and does not begin to hint at
its actual value. For example, while  Kuiper's theorem \cite{kuiper} implies   that  $m(S^2\!\times\! S^2)$  cannot admit a 
metric with $W_+\equiv 0$, 
 it doesn't  guarantee that $\inf \int | W_+|^2 d\mu$ of such a manifold could never equal zero. Proving  that 
this infimum is actually positive would  be an interesting  accomplishment in itself!

Our lack of effective estimates for   $m_0$ becomes even more severe in the non-spin setting of 
 Theorem \ref{madhatter}. Given a closed interval 
$I \subset ( \frac{4}{5}, 2)$, we have seen that there  is an integer $m_0$ so that a metric 
satisfying \eqref{rabbit} can be found on $m\CP_2 \# n \overline{\CP}_2$ whenever $m\geq m_0$ and $\frac{n}{m} \in I$. 
However,  the value of $m_0$ produced by the proof is   astronomical in practice, and in any case 
tends to infinity when, for example, the  lower endpoint of $I$ approaches $\frac{4}{5}$. 
Another possible objection 
is that  inequality \eqref{rabbit} depends on a choice of orientation. This, however, is  not really a serious issue, 
because the construction produces metrics that satisfy \eqref{rabbit} for {\sf both} orientations if 
$I\subset (\frac{4}{5}, \frac{5}{4})$. 

Finally, the almost-K\"ahler version of our question of balance has only been touched on  here in a very preliminary way,
and there could  be many interesting things that remain to be discovered  in this setting. 
For example, while we have seen that the direction in which the balance tips is different for 
two interesting classes of 
almost-K\"ahler manifolds, it is possible that  the patterns we have noticed may  hold for larger
classes 
almost-K\"ahler metrics. For example, 
if an almost-K\"ahler metric has positive scalar curvature, 
it then follows that $c_1 \cdot [\omega ] > 0$;  and, conversely,  this condition certainly suffices to imply 
 that $W_+$ has relatively large $L^2$-norm. Is there a version of the almost-K\"ahler  
 balance story  that only depends on the sign of $c_1\cdot [\omega ]$? A clean result  along these lines
 would certainly shed interesting new light on the subject. 

\pagebreak 


\end{document}